\documentclass[12pt]{amsart}
\usepackage[margin=1.5in]{geometry}
\usepackage{amsmath, amssymb, amsfonts}
\usepackage[foot]{amsaddr}
\usepackage{color}

\newcommand{\R}{\mathbb{R}}
\newcommand{\C}{\mathbb{C}}

\newcommand{\Z}{\mathbb{Z}}

\newcommand{\x}{\mbox{\boldmath $x$}}
\newcommand{\y}{\mbox{\boldmath $y$}}

\newtheorem{theorem}{Theorem}

\newtheorem{corollary}[theorem]{Corollary}

\begin{document}

\title{Large sets of complex and real equiangular lines}

\author{Jonathan Jedwab}
\address{Department of Mathematics, Simon Fraser University, 8888 University Drive, Burnaby, BC, Canada, V5A 1S6}
\author{Amy Wiebe}
\address{Department of Mathematics, University of Washington, Box 354350, Seattle, WA 98195-4350, USA}
\date{20 January 2015 (revised 7 March 2015)}
\email{jed@sfu.ca, awiebe@math.washington.edu}

\begin{abstract} 
Large sets of equiangular lines are constructed from sets of mutually unbiased bases, over both the complex and the real numbers.
\end{abstract}

\maketitle

\section{Introduction} \label{sec:intro}
The \emph{angle} between vectors $\x_j$ and $\x_k$ of unit norm in $\C^d$ is $\arccos|\langle \x_j,\x_k\rangle|$, where $\langle \cdot , \cdot \rangle$ is the standard Hermitian inner product. 
A set of $m$ distinct lines in $\C^d$ through the origin, represented by vectors $\x_1,\ldots,\x_m$ of equal norm, is \emph{equiangular} if for some real constant $a$ we have
\[
|\langle \x_j,\x_k\rangle| = a \quad \mbox{for all $j\ne k$.}
\]
The number of equiangular lines in $\C^d$ is at most $d^2$~\cite{DGS-bounds}, and when the vectors are further constrained to lie in $\R^d$ this number is at most $d(d+1)/2$ (attributed to Gerzon in \cite{Lemmens}). It is an open question, in both the complex and real case, whether the upper bound can be attained for infinitely many~$d$, although in both cases $\Theta(d^2)$ equiangular lines exist for all~$d$. Specifically, K{\"o}nig \cite{konig} constructed $d^2-d+1$ equiangular lines in $\C^d$ where $d-1$ is a prime power, and de Caen \cite{deCaen} constructed $2(d+1)^2/9$ equiangular lines in $\R^d$ where $(d+1)/3$ is twice a power of~4. By extending vectors using zero entries as necessary, we can derive sets of $\Theta(d^2)$ equiangular lines from these direct constructions for all~$d$.

Two orthogonal bases $\{\x_1,\ldots, \x_d\},\{\y_1,\ldots,\y_d\}$ for $\C^d$ are \emph{unbiased} if 
\begin{equation} \frac{|\langle \x_j,\y_k\rangle|}
{||\x_j||\cdot||\y_k||} = \frac{1}{\sqrt{d}} \label{eqn:MUBdefn} 
\quad \mbox{for all $j,k$}.
\end{equation}
A set of orthogonal bases is a set of \emph{mutually unbiased bases} (MUBs) if all pairs of distinct bases are unbiased.

The number of MUBs in $\C^d$ is at most $d+1$ \cite[Table I]{DGS-bounds}, which can be attained when $d$ is a prime power by a variety of methods \cite{Godsil-Roy}, \cite{Kantor-inequivalence}, \cite{Wootters}. The number of MUBs in $\R^d$ is at most $d/2+1$ \cite[Table I]{DGS-bounds}, which can be attained when $d$ is a power of~4 \cite{Calderbank-Z4-Kerdock}, \cite{Cameron-quadratic}.

The authors recently gave a direct construction of $d^2/4$ equiangular lines in $\C^d$, where $d/2$ is a prime power~\cite{Jedwab-Wiebe-MUBs-lines}. We show here how to generalize the underlying construction to give $\Theta(d^2)$ equiangular lines in $\C^d$ and $\R^d$ directly from sets of complex and real MUBs.

\section{The Construction}\label{S:Lblocks}
We associate an ordered set of $m$ vectors in $\C^d$ with the $m \times d$ matrix formed from the vector entries, using the ordering of the set to determine the ordering of the vectors.

\begin{theorem}
\label{thm:concatenation}
Suppose that $B_1, B_2, \dots, B_r$ form a set of $r$ MUBs in $\C^d$, each of whose vectors has all entries of unit magnitude, where $r \le d$. Let $a_1, a_2, \dots, a_t$ be constants in $\C$, where $t \ge 1$.
Let $B_j(v)$ be the set of $d$ vectors formed by multiplying entry $j$ of each vector of $B_j$ by $v \in \C$, and let $L(v) = \cup_{j=1}^r B_j(v)$ (considered as an ordered set).
Then all inner products between distinct vectors among the $rd$ vectors of 
\[
\Big [ L(a_1) \quad L(a_2) \quad \dots \quad L(a_t) \quad L\Big(t+1-\sum_{j=1}^ta_j\Big) \Big ]
\]
in $\C^{(t+1)d}$ have magnitude $\sum_{j=1}^t\left|a_j-1\right|^2 + \left|\sum_{j=1}^t(a_j-1)\right|^2$ or $(t+1)\sqrt{d}$.
\end{theorem}

\begin{proof} 
Write $A = \{a_1, a_2, \dots, a_t, t+1-\sum_{j=1}^t a_j\}$ for the set of arguments~$v \in \C$ taken by $L(v)$ in the construction.
We consider two cases, according to whether distinct vectors of $L(v)$ originate from the same basis or from distinct bases. 

In the first case, consider the inner product of distinct vectors of $L(v)$ constructed from vectors from the same basis $B_j$. Since the original vectors are orthogonal, this inner product is $z(|v|^2-1)$ for some $z$ of unit magnitude that depends only on the original two vectors. Since each occurrence of $L(v)$ uses the same ordering, the inner product of the corresponding concatenated vectors in $\C^{(t+1)d}$ is therefore $z \sum_{v\in A} (|v|^2-1)$, which equals $z\Big(\sum_{j=1}^t\left|a_j-1\right|^2 + \left|\sum_{j=1}^t(a_j-1)\right|^2\Big)$ after straightforward algebraic manipulation.

In the second case, consider vectors of $L(v)$ constructed from vectors from distinct bases $B_j, B_k$ . Let these constructed vectors be 
\[
\begin{array}{rcccccc}
\x = (x_{1} & x_2 & \ldots & vx_j   & \ldots & \ldots & x_d),\\
\y = (y_{1} & y_2 & \ldots & \ldots & vy_k   & \ldots & y_d). 
\end{array}
\]
The inner product of $\x$ and $\y$ in $L(v)$ is
\[
x_1\overline{y_1} + \cdots + vx_j\overline{y_j} +\cdots+\overline{v}x_k\overline{y_k}+\cdots+x_d\overline{y_d} 
= \sum_{\ell=1}^d x_\ell\overline{y_\ell} + (v-1)x_j\overline{y_j}+(\overline{v}-1)x_k\overline{y_k}.
\]
Therefore the corresponding concatenated vectors in $\C^{(t+1)d}$ have inner product 
\[
(t+1)\sum_{\ell=1}^d x_\ell\overline{y_\ell} + x_j\overline{y_j} \sum_{v\in A} (v-1) + x_k\overline{y_k} \sum_{v\in A}(\overline{v}-1) 
= (t+1)\sum_{\ell=1}^d x_\ell\overline{y_\ell},
\]
because $\sum_{v \in A}v = t+1$. 
Now, all of the entries $x_\ell$, $y_\ell$ have unit magnitude by assumption, and so $\left|\sum_{\ell=1}^d x_\ell\overline{y_\ell}\right| = \sqrt{d}$ by the MUB property~\eqref{eqn:MUBdefn}. Therefore the concatenated vectors in $\C^{(t+1)d}$ have inner product of magnitude $(t+1)\sqrt{d}$.
\end{proof}

\emph{Remark.}
Lemma~6.2 of \cite{Jedwab-Wiebe-MUBs-lines} describes the special case $t=1$ and $r=d$ of Theorem~\ref{thm:concatenation}, in which the MUBs are constrained to arise from a $(d,d,d,1)$ relative difference set in an abelian group according to the construction method of~\cite{Godsil-Roy}; the permutation $\pi$ given in \cite[Lemma~6.2]{Jedwab-Wiebe-MUBs-lines} can be dropped without loss of generality.

\begin{corollary}
\label{cor:complex-lines}
Let $t$ be a positive integer and let $d$ be a prime power. 
There exist $d^2$ equiangular lines in $\C^{(t+1)d}$.
\end{corollary}
\begin{proof}
There exists a set of $d+1$ MUBs in $\C^d$ for which one of the bases is the standard basis~\cite{Wootters}. After appropriate scaling, all entries of each of the vectors of the remaining $d$ bases therefore have unit magnitude, using~\eqref{eqn:MUBdefn}. So we may apply Theorem~\ref{thm:concatenation} with $r=d$. 

There are infinitely many choices of $a_1, a_2, \dots, a_t \in \C$ for which the two magnitudes in the conclusion of Theorem~\ref{thm:concatenation} are equal, one such choice being $a_j = 1+d^{1/4}/\sqrt{t}$ for each~$j$.
\end{proof}

\begin{corollary}
\label{cor:real-lines}
Let $t$ be a positive integer and let $d$ be a power of~$4$.
There exist $d^2/2$ equiangular lines in $\R^{(t+1)d}$. 
\end{corollary}
\begin{proof}
There exists a set of $d/2+1$ MUBs in $\R^d$ for which one of the bases is the standard basis \cite{Calderbank-Z4-Kerdock}, \cite{Cameron-quadratic}. Apply Theorem~\ref{thm:concatenation} with $r=d/2$ and take, for example, $a_j = 1+d^{1/4}/\sqrt{t}$ for each~$j$ to obtain real equiangular lines.
\end{proof}

The proof of Theorem~\ref{thm:concatenation} shows that the magnitude of the inner product of distinct vectors is $\sum_{v\in A}(|v|^2-1)$ or $(t+1)\sqrt{d}$. In the construction of Corollaries~\ref{cor:complex-lines} and~\ref{cor:real-lines}, the constants $a_j$ are chosen so that these magnitudes are equal, and the inner product of each concatenated vector with itself is $\sum_{v\in A} (|v|^2+d-1)$. It follows that the common angle for the sets of equiangular lines constructed in Corollaries~\ref{cor:complex-lines} and~\ref{cor:real-lines} is $\arccos(1/(1+\sqrt{d}))$ for all $t$, regardless of the choice of the constants~$a_j$. 

Theorem~\ref{thm:concatenation} can be generalized as follows. Let $c_1, \dots, c_t$ be real constants, and take the $rd$ vectors of
\[
\Big [ c_1L(a_1) \quad c_2L(a_2) \quad \dots \quad c_tL(a_t) \quad L\Big(1+\sum_{j=1}^tc_j^2(1-a_j)\Big) \Big ]
\]
in $\C^{(t+1)d}$.
Then all inner products between distinct vectors have magnitude $\sum_{j=1}^t c_j^2\left|1-a_j\right|^2 + \left|\sum_{j=1}^t c_j^2(1-a_j)\right|^2$ or $(1+\sum_{j=1}^tc_j^2) \sqrt{d}$.
If $a_1,a_2,\dots,a_t$ and $c_1,c_2,\dots,c_t$ are chosen so that these two magnitudes are equal, the common angle of the resulting set of equiangular lines is again $\arccos(1/(1+\sqrt{d}))$.

\emph{Remark.}
The case $t=1$ and $d=4$ of Corollary~\ref{cor:real-lines} constructs 8 equiangular lines in $\R^8$ having the form $\left[ L(a) \quad L(2-a) \right]$, where $a \in \{1 \pm \sqrt{2}\}$. We can extend this to a set 
$\left[ \begin{smallmatrix} L(a) & L(2-a) \\ L(2-a) & L(a) \end{smallmatrix} \right]$
of 16 equiangular lines in $\R^8$, where $a \in \{1 \pm \sqrt{2}\}$; this extension does not seem to generalize easily to larger values of~$d$.

\section*{Acknowledgements}
J. Jedwab is supported by an NSERC Discovery Grant.


\end{document}